\documentclass[11pt,letterpaper]{amsart}
\usepackage{amsmath}
\usepackage{amssymb}
\usepackage{amsthm}
\usepackage{hyperref}
\usepackage[noabbrev,capitalize]{cleveref}
\usepackage{mathrsfs}
\usepackage{mathtools}
\usepackage{slashed}
\usepackage{tikz-cd}
\usepackage[shortlabels]{enumitem}

\setenumerate{label=(\roman*)}

\DeclareMathOperator{\Ric}{Ric}

\def\sideremark#1{\ifvmode\leavevmode\fi\vadjust{\vbox to0pt{\vss
 \hbox to 0pt{\hskip\hsize\hskip1em
 \vbox{\hsize3cm\tiny\raggedright\pretolerance10000
 \noindent #1\hfill}\hss}\vbox to8pt{\vfil}\vss}}}

\newtheorem{theorem}{Theorem}[section]
\newtheorem{proposition}[theorem]{Proposition}
\newtheorem{lemma}[theorem]{Lemma}

\theoremstyle{definition}
\newtheorem{definition}[theorem]{Definition}

\theoremstyle{remark}
\newtheorem{remark}[theorem]{Remark}

\numberwithin{equation}{section}
\usepackage{ stmaryrd }

\allowdisplaybreaks
\makeatletter
\@namedef{subjclassname@2020}{\textup{2020} Mathematics Subject Classification}
\makeatother

\begin{document}
\title{Spectral splitting theorem and ends of minimal hypersurfaces}
\author{Han Hong}
\address{Department of Mathematics and statistics \\ Beijing Jiaotong University \\ Beijing \\ China, 100044}
\email{hanhong@bjtu.edu.cn}

\author{Gaoming Wang}
\address{Beijing Institute of Mathematical Sciences and Applications \\ Beijing \\ China, 100044}
\email{wanggaoming@bimsa.cn}
\date{}
\maketitle

\begin{abstract}
    In this paper, we give a new proof of the splitting theorem on manifolds with nonnegative spectral Ricci curvature proved in \cite{antonelli-pozzetta-xu,catino-Mari-Mastrolia-Roncoroni,hong-wang-jfa}. Furthermore, by constructing weighted minimizing geodesics near infinity, we show that minimal hypersurfaces with finite index in $n$-manifolds, $n\leq 5$, with nonnegative biRic curvature must have finite ends, generalizing a result of Li--Wang \cite{Li-Wang-nonnegatively-curved-manifold} on manifolds with nonnegative sectional curvature.
\end{abstract}

\section{Introduction}

Let $(M,g)$ be a noncompact $n$-dimensional Riemannian manifold, and denote by $\eta$ the outward unit normal vector along $\partial M$ if $M$ has nonempty boundary. Define the Lipschitz function $\Ric : M \to \mathbb{R}$ by
\[
\Ric(x)=\inf\{\Ric(v,v): g(v,v)=1,\ v\in T_xM\}.
\]
Given a constant $\alpha\geq 0$, we say that $M$ has \emph{nonnegative $\alpha$-Ricci curvature in the spectral sense} or \emph{nonnegative spectral $\alpha$-Ricci curvature} if
\[
\lambda_1(-\alpha\Delta+\Ric)\geq 0.
\]

Equivalently,
\[
\int_M \alpha |\nabla \varphi|^2+\Ric\,\varphi^2\geq 0
\]
for every compactly supported function $\varphi$. By standard elliptic theory, this is also equivalent to the existence of a positive function $u\in C^{2,\beta}(M)$ satisfying
\begin{equation}\label{equation 1 of u}
-\alpha\Delta u+\Ric\,u=0
\quad\text{on }M.
\end{equation}
If $M$ has nonempty boundary, we additionally have
\begin{equation}\label{equation 2 of u}
\frac{\partial u}{\partial \eta}=0
\quad\text{on }\partial M.
\end{equation}

Spectral Ricci curvature has attracted considerable attention in recent years. Among the main developments are spectral versions of the Bonnet--Myers theorem and the Bishop--Gromov volume comparison theorem \cite{antonelli-xu}, as well as spectral splitting theorems \cite{antonelli-pozzetta-xu,catino-Mari-Mastrolia-Roncoroni,hong-wang-jfa}.

\begin{theorem}[{\cite{antonelli-pozzetta-xu,catino-Mari-Mastrolia-Roncoroni}}]\label{antonelli}
Let $n\geq 2$, let $\alpha<\frac{4}{n-1}$, and let $(M^n,g)$ be a complete noncompact Riemannian manifold without boundary satisfying
\[
\lambda_1(-\alpha\Delta+\Ric)\geq 0.
\]
If $M$ has at least two ends, then $\Ric\geq 0$ on $M$, and consequently $M$ splits isometrically as $N\times \mathbb{R}$ where $N$ is a closed manifold with nonnegative Ricci curvature.
\end{theorem}

Although their work appeared on the same day, Antonelli--Pozzetta--Xu proved Theorem \ref{antonelli} using the theory of $\mu$-bubbles, while Catino--Mari--Mastrolia--Roncoroni employed the critical theory of Schrödinger operators. Moreover, Antonelli--Pozzetta--Xu showed, by constructing a counterexample, that the assumption of having at least two ends cannot be replaced by the existence of a minimizing line. Later, the authors applied the second variation formula for weighted length together with Liu's deformation argument (see \cite{liugang}) to establish the following boundary analogue.

\begin{theorem}[{\cite{hong-wang-jfa}}]\label{hongwang}
Let $(M^n,g)$ be a smooth noncompact Riemannian manifold with mean-convex boundary and nonnegative $\alpha$-Ricci curvature in the spectral sense. Assume that $\alpha<\frac{4}{n-1}$ and that there exists a free-boundary weighted minimizing ray with respect to the weighted length functional $\int u^\alpha$. Then $\Ric\geq 0$ on $M$, and $M$ splits off an $\mathbb{R}_+$ factor.
\end{theorem}

We remark that the argument in \cite{hong-wang-jfa} also yields the following statement, showing that the condition of two ends in Theorem \ref{antonelli} can be replaced by a weaker condition (see Proposition \ref{existence of weighted minimizing line}).

\begin{theorem}\label{main theorem 1}
Let $(M^n,g)$ be a smooth complete noncompact Riemannian manifold with nonnegative $\alpha$-Ricci curvature in the spectral sense. Assume that $\alpha<\frac{4}{n-1}$ and that there exists a weighted minimizing line with respect to $\int u^\alpha$. Then $\Ric\geq 0$ on $M$, and hence $M$ splits off an $\mathbb{R}$ factor.
\end{theorem}

A common feature of the proofs in \cite{antonelli-pozzetta-xu,catino-Mari-Mastrolia-Roncoroni,hong-wang-jfa} is that the final step establishes pointwise nonnegativity of the Ricci curvature. The first goal of this paper is to provide a fourth proof of these splitting results by adapting the classical argument of the Cheeger--Gromoll splitting theorem to the spectral setting (see Section \ref{section2}).

\vskip.2cm
An important application of Theorem \ref{antonelli} concerns the topology of stable minimal hypersurfaces (not necessarily proper) in manifolds with nonnegative biRic curvature. In particular, such hypersurfaces have at most two ends; see \cite[Corollary 1.2]{antonelli-pozzetta-xu}. This motivates the second main result of the present paper below.

\vskip0.2cm

Spectral curvature conditions arise naturally in the study of stable minimal hypersurfaces and play a fundamental role in recent progress on stable Bernstein theorems \cite{chodosh-li-stryker,mazet}. In particular, if the ambient manifold $N$ has nonnegative biRic curvature, then every stable minimal hypersurface admits nonnegative spectral Ricci curvature. Indeed, combining the Gauss equation with the stability inequality yields
\begin{equation}\label{spectral Ric}
\int_M |\nabla \varphi|^2+\Ric_M\varphi^2
\geq C\int_M \varphi^2
\end{equation}
whenever the biRic curvature of $N$ is bounded below by $C$. This observation first appeared in the work of Shen and Ye \cite{shenyingyerugang} and can be viewed as a natural generalization of the classical theorem of Fischer-Colbrie and Schoen \cite{Fischer-Colbrie-Schoen-The-structure-of-complete-stable}.

The study of the ends of stable minimal hypersurfaces is closely related to Bernstein-type problems. In Euclidean space, complete stable minimal hypersurfaces have only one end \cite{Cao-Shen-Zhu-infinitevolume}. Li and Wang later proved that finite-index minimal hypersurfaces in $\mathbb{R}^n$ have finitely many ends \cite{Li-Wang-finiteindex}, using Sobolev inequalities on minimal hypersurfaces. Subsequently, Li and Wang \cite{Li-Wang-nonnegatively-curved-manifold} showed that properly immersed finite-index minimal hypersurfaces in manifolds with nonnegative sectional curvature also have finitely many ends. Their argument relies on the convexity of Busemann functions under nonnegative sectional curvature and only detects ends extending to infinity, which explains the necessity of the properness assumption.

The second goal of this paper is to generalize the result of Li--Wang to minimal hypersurfaces (not necessarily proper) in manifolds satisfying weaker curvature assumptions.

\begin{theorem}\label{main theorem 2}
Let $(M^n,g)$, $n\leq 5$, be a Riemannian manifold with nonnegative biRic curvature. Then every finite-index complete noncompact two-sided minimal hypersurface in $M$ has finitely many ends.
\end{theorem}

Indeed, if $\Sigma$ is a finite-index minimal hypersurface in $M$, then $\Sigma$ is stable outside a compact subset $\Omega$. By \eqref{spectral Ric}, the manifold $\Sigma\setminus\Omega$ has nonnegative spectral $1$-Ricci curvature. Equivalently, there exists a positive function $u$ on $\Sigma\setminus\Omega$ satisfying
\[
-\Delta u+\Ric\,u=0
\quad\text{on }\Sigma\setminus\Omega.
\]
There always exists a free boundary weighted minimizing ray on $\Sigma\setminus\Omega$ since $\partial(\Sigma\setminus \Omega)$ is compact. If $\partial(\Sigma\setminus\Omega)$ is mean convex, Theorem \ref{hongwang} would complete the proof of Theorem \ref{main theorem 2}. However, this is not ensured no matter how large we choose $\Omega$ to be. Thus the goal is not to find free boundary rays but instead lines on $\Sigma\setminus \Omega.$

The proof of Theorem \ref{main theorem 2} proceeds by contradiction and consists of three steps. First, assuming that $\Sigma$ has infinitely many ends, we construct a weighted minimizing geodesic line at infinity. Second, we perturb the above function $u$ to obtain a positive function, still denoted by $u$, and a compact subset $K\subset\Sigma$ such that
\[
-\Delta u+\Ric\,u>0
\quad\text{on }\Sigma\setminus K.
\]
Finally, applying the second variation formula for the weighted length functional yields rigidity, leading to the desired contradiction with the strict inequality above.

During the completion of this work, Zhu-Bi \cite{zhu-bi} proved the finiteness of ends for manifolds with nonnegative Ricci curvature outside a compact set, also using geodesic lines at infinity. Ours can be thought of as a spectral version of their result. It is also worth noting that recent work by Antonelli, Li, and Sweeney \cite{antonelli-li} proved a spectral splitting theorem for manifolds with boundary, improving $\alpha$ to a larger 
 bound $\frac{n-1}{n-2}$ under different assumptions.

\vskip.2cm
 In Section \ref{section2}, we provide a new proof of the spectral splitting theorem. In Section \ref{section3}, we show that nonnegative spectral $\alpha$-Ricci curvature at infinity together with infinitely many ends implies the existence of a weighted minimizing line; we then prove the main result. 

 \subsection{Acknowledgements}
 The first author is supported by NSFC No. 12401058 and the Fundamental Research Funds for the Central Universities No. 2024XKRC008.

\section{Spectral splitting}\label{section2}

In this section, we give a new and direct proof of the spectral splitting results in \cite{antonelli-pozzetta-xu,catino-Mari-Mastrolia-Roncoroni,hong-wang-jfa}.

Let $M$ be a complete noncompact manifold and $u$ be a positive $C^2$ function on $M$. For a piecewise smooth curve $\gamma$, we consider the weighted length function defined by
\[
	L_u^\alpha(\gamma)=\int_{ \gamma}u^\alpha ds 
\]
and for two points $p,q \in M$, we define the weighted distance as

\[
	L_u^\alpha(p,q)=\inf_{\gamma \text{ has endpoints }p,q} L_u^\alpha(\gamma).
\]

\begin{definition}
    We say $(M,g)$ contains a weighted minimizing line (ray) if there exists a curve $\gamma(s):(-\infty,\infty)\rightarrow M$ (resp. $\gamma: [0,\infty)\rightarrow M$) parametrized by the $g$-arclength such that the $L_u^\alpha$-length of an arc of $\gamma$ between any two points is just the weighted distance between these two points. In other words,
    \[L_u^\alpha(\gamma(t_1),\gamma(t_2))=L_u^\alpha(\gamma_{t_1,t_2}), \ \ \forall\ t_1,t_2\in (-\infty,\infty), \ \ (\text{resp}. \ t_1,t_2\in [0,\infty))\]
    where $\gamma_{t_1,t_2}$ is the arc of $\gamma$ between $\gamma(t_1)$ and $\gamma(t_2)$.
\end{definition}
\begin{remark}
    For a weighted minimizing line $\gamma$ in $M$, $L_u^\alpha(\gamma)$ is possibly finite.
\end{remark}
In the following we prove a spectral Laplacian comparison theorem.
\begin{lemma}
Suppose $(M,g)$ has nonnegative spectral Ricci curvature, i.e., there exists a positive function $u$ such that $-\alpha \Delta u +\mathrm{Ric}\cdot u\ge 0.$ Let $p$ be a fixed point in $M$ and $\alpha<\frac{4}{n-1}$, 	
	then we have
	\[
		\Delta_q L_u^\alpha(p,q)\le \frac{Cu^\alpha(q)}{d(p,q)}
	\]
	in the viscosity sense.
	Here, $C=C(n,\alpha)$ is a constant depending only on $n,\alpha.$
	
\end{lemma}
\begin{proof}
	We assume $p,q$ can be connected by a weighted minimizing curve (parametrized by $g$-arclength) $\gamma:[0,l]\rightarrow M$. 
	We also assume $q$ is not a cut point of $p$ along $\gamma$ for simplicity.

	Let $V$ be the variational vector field defined on a neighborhood of $\gamma(s)$.
	Let $\gamma_t(s)$ be the variation of $\gamma$ with variational vector field $V$.
	Then it follows from the second variation formula (see \cite[Lemma 2.4]{hong-wang-jfa}) that
	\begin{align*}
	    \delta^2_VL_u^\alpha(\gamma)&=\int_{ \gamma}|(\nabla_{T}V)^{\perp}|^{2}u^\alpha-R(T,V,T,V)u^\alpha+2 \mathrm{div}_\gamma(V)\nabla_{V}u^\alpha+\nabla^{2}u^\alpha(V,V)\\
        &+u^\alpha \left< \nabla_{V}V,\partial_s \right> |^{\gamma(l)}_{\gamma(0)}.
	\end{align*}
		
	If we consider $\gamma_t(0)=p$ for all $t$, and we define the function $f(t)=L_u^\alpha(\gamma_t)$.
	Then, we have $L_u^\alpha(p,\gamma_t(l))\le f(t)$ and equality holds when $t=0$.
	Then, we have
	\begin{align*}
		\nabla_V \nabla_VL_u^\alpha(p,q)={} & \frac{\partial^2 L_u^\alpha(p,\gamma_t(l))}{\partial t^2}-\left< \nabla_{V}V,\nabla L_u^\alpha(p,q) \right> |_{\gamma(l)}\\
		\le{}& \frac{\partial ^2f}{\partial t^2}-u^\alpha\left< \nabla_{V}V,\partial_s \right> |_{\gamma(l)}\\
		={}& \int_{ \gamma}(|(\nabla_{T}V)^{\perp}|^{2}-R(T,V,T,V))u^\alpha+2 \mathrm{div}_\gamma(V)\nabla_{V}u^\alpha+\nabla^{2}u^\alpha(V,V)
	\end{align*}
	
	For a $C^1$ function $\phi$ defined on $\gamma$ with $\phi(0)=0, \phi(l)=1$, we consider the vector field $V=\phi e_i$ along $\gamma$.
	Here, $e_n=\gamma'(s)$, and $\left\{ e_i \right\}$ is an orthonormal parallel frame normal to $\gamma'$.
	After taking summation, we have
	\begin{align*}
\Delta_q L_u^\alpha(p,q)\le{}&\int_{ \gamma}(n-1)\phi_s^{2}u^\alpha - \mathrm{Ric}(\gamma',\gamma')\phi^{2}u^\alpha  - \alpha\phi^{2}u^\alpha|\nabla^\bot \log u|^{2}+2\phi_s\phi \nabla_s u^\alpha+\phi^{2}\Delta u^\alpha 
	\end{align*}
We rewrite $\phi \rightarrow \psi u ^{-\frac{\alpha}{2}}$ and after simplification, we have
\begin{align}
\Delta_qL_u^\alpha(p,q)\le{}&\int_{ \gamma}(n-1)\psi_s^{2}-\mathrm{Ric}(\gamma',\gamma')\psi^{2}-\alpha \psi^{2} |\nabla^\bot \log u|^{2}\nonumber\\
	&+\alpha\left( \frac{n-1}{4}\alpha-1 \right)\psi^{2} \frac{u_s^{2}}{u^{2}}-\alpha(n-3)\psi \psi_s \frac{u_s}{u}+\alpha \psi^{2} \frac{\Delta u}{u}\nonumber\\
    \le{}&\int_{ \gamma}(n-1)\psi_s^{2}- \alpha \psi^{2}|\nabla^\bot \log u|^{2}+\alpha\left( \frac{n-1}{4}\alpha-1 \right)\psi^{2} \frac{u_s^{2}}{u^{2}}+\alpha(n-3)|\psi \psi_s| \frac{|u_s|}{u}\nonumber\\
    \le{}&\int_{ \gamma}C\psi_s^{2}\nonumber
\end{align}
for some constant $C=C(n,\alpha)$ when $\alpha<\frac{4}{n-1}$.
Here, we have used the inequality
\[
	-\alpha \Delta u +\mathrm{Ric}\cdot u\ge 0.
\]
Now, if we choose $\psi=\frac{u^{\frac{\alpha}{2}}(\gamma(l))}{l}s$,
we have
\[
	\Delta_q L_u^\alpha(p,q)\le \frac{Cu^\alpha(q)}{l}\le\frac{Cu^\alpha(q)}{d(p,q)}.
\]

Now, we assume the minimizing curve is achieved by $\gamma=\gamma_q\cup \cdots\cup \gamma_i\cup\cdots \cup \gamma_p$ where $\gamma_q:[0,+\infty)\rightarrow M$, $\gamma_p:(-\infty,0]\rightarrow M$ are minimizing rays and each $\gamma_i:(-\infty,\infty)\rightarrow M$ is a weighted minimizing line, moreover, $\gamma_q(0)=q, \gamma_p(0)=p$.
Note that $L_u^\alpha(p,q')\le L_u^\alpha(\gamma_1(t),q')+L_u^\alpha(\gamma_1(t),p)$ for any $t>0$ and $q'\in M$, with equality when $q'=q$.
Then we find
\[
	\Delta_{q} L_u^\alpha(p,q) \le \Delta_q L_u^\alpha(\gamma_1(t),q)\le \frac{Cu^\alpha(q)}{t}.
\]
We can take $t\rightarrow \infty$ to finish the proof.
\end{proof}

We shall use the weighted Busemann function to complete the proof of Theorem \ref{main theorem 1}.

\begin{proof}[Proof of Theorem \ref{main theorem 1}]
We define two weighted Busemann functions
\begin{align*}
	b^+(p)={} & \lim_{t\rightarrow +\infty}L_u^\alpha(\gamma(t),\gamma(0))-L_u^\alpha(\gamma(t),p)\\
	b^-(p)={} & \lim_{t\rightarrow +\infty}L_u^\alpha(\gamma(-t),\gamma(0))-L_u^\alpha(\gamma(-t),p).
\end{align*}
Let $b_t^+=L_u^\alpha(\gamma(t),\gamma(0))-L_u^\alpha(\gamma(t),p)$ and $b_t^-=L_u^\alpha(\gamma(-t),\gamma(0))-L_u^\alpha(\gamma(-t),p)$. It is easy to check that
\begin{itemize}
    \item $b_t^+(p)$ is uniformly bounded in $t$ on compact subsets of $M$, because
    \[|b_t^+(p)|\leq L_u^\alpha(\gamma(0),p).\]
    \item $b_t^+(p)$ is monotone increasing in $t$, because for $0<t_1<t_2$,
    \[b_{t_2}^+-b_{t_1}^+=L_u^\alpha(\gamma(t_1),p)-L_u^\alpha(\gamma(t_2),p)+L_u^\alpha(\gamma(t_1),\gamma(t_2))\geq  0.\]
    \item $b_t^+(p)$ is Lipschitz, because for $p,q\in M$,
    \[|b_t^+(p)-b_t^+(q)|=|L_u^\alpha(\gamma(t),p)-L_u^\alpha(\gamma(t),q)|\leq L_u^\alpha(p,q)\leq C(p,q)d(p,q)\]
    where $\lim_{q\rightarrow p}C(p,q)=u^{\alpha}(p).$
\end{itemize}
Similar properties hold for $b_t^-$. Note that $d(\gamma(t),q)$ tends to infinity as $t$ tends to infinity, so for any compactly supported function $\varphi\geq 0$ we have
\[\int_M b_t^+ \Delta \varphi=\int_M\varphi\Delta b_t^+=-\int_M \varphi \Delta L_u^\alpha(\gamma(t),p)\geq -\int_M\frac{C\varphi u^\alpha(p)}{d(\gamma(t),p)}\longrightarrow 0\]
in the distribution sense as $t\rightarrow \infty$.

Thus we can obtain that
\[
	\Delta b^+(p)\ge 0.
\]
Similarly,
\[\Delta b^-(p)\ge 0.\]
In general, we do not know whether $L_u^\alpha(\gamma(t),\gamma(0))$ as $t\rightarrow \pm\infty$ is finite or not.
Nevertheless, the triangle inequality implies $b_t^+(p)+b_t^-(p)\le 0$, thus
\[
	b^+(p)+b^-(p)\le 0.
\]
Moreover, we have equality along the line $\gamma$.
By the maximum principle, we have
\[
	b^+(p)+b^-(p)\equiv 0, \quad \Delta b^+(p)= \Delta b^-(p)\equiv 0.
\]
Furthermore, by elliptic regularity, $b^+$ and $b^-$ are smooth functions. 

Now, we apply the Bochner formula to obtain
\[
	\frac{1}{2}\Delta |\nabla b^+|^{2}={} \left| \nabla^{2} b^+ \right|^{2}+\mathrm{Ric}(\nabla b^+,\nabla b^+).
\]
Note that $|\nabla b^+|^2=u ^{2\alpha}$. Then
\[
	\frac{1}{2}\Delta |\nabla b^+|^{2}=\alpha u^{2\alpha -1}\Delta u +\alpha(2\alpha -1)u^{2\alpha -2}|\nabla u|^{2}.
\]
It follows that
\[
	|\nabla^2b^+|^2-\frac{2\alpha-1}{\alpha}|\nabla u^\alpha|^{2}=\alpha u^{2\alpha -1}\Delta u -\mathrm{Ric}|\nabla b^+|^2\le 0
\]
which is
\[
	|\nabla^2b^+|^2\le\frac{2\alpha-1}{\alpha}|\nabla |\nabla b^+||^{2}.
\]
By Kato's inequality, we have
\[
	|\nabla|\nabla b^+||^{2}\le \frac{n-1}{n}|\nabla^2 b^+|^2.
\]
So
\[
	|\nabla^2b^+|^2\le \frac{(2\alpha-1)(n-1)}{n\alpha}|\nabla^2 b^+|^2.
\]
Note that when $\alpha<\frac{n-1}{n-2}$,
we have
\[
	\frac{(2\alpha-1)(n-1)}{n\alpha}<1.
\]
Hence $\nabla^2 b^+=0$, $u$ is a constant function, and splitting follows.
\end{proof}

The following proposition says that a weighted minimizing line is guaranteed when $M$ has at least two ends.
\begin{proposition}\label{existence of weighted minimizing line}
    Let $\alpha<\frac{4}{n-1}$ and $u$ be a $C^2$ positive function on $M$. Let $(M, g)$ be a smooth noncompact Riemannian manifold. If $M$ has at least two ends, then there exists a weighted minimizing line in $M$.
\end{proposition}

\begin{proof}
    Fix a point $p \in M$. Since $M$ has at least two ends, we can choose a sufficiently large $R_0$ such that $M \setminus B_{R_0}(p)$ has at least two unbounded components, denoted by $E_{R_0}^1,E_{R_0}^2$, respectively. 

    Now, consider a family of exhaustions $B_{R}(p)$ of $M$. Without loss of generality, we assume that $\partial B_R(p)$ is smooth. We also choose a smooth perturbation of the weighted function $u_R$ such that:
    \begin{itemize}
        \item $u_R = u$ in $B_{R}(p)$,
        \item $u_R \geq u$ on $M$,
        \item $u_R \geq  \max\{1, u\}$ in $M \setminus B_{R+1}(p)$.
    \end{itemize}
    Since $\partial B_{R}(p) \cap E_{R_0}^i$ is compact, there exists a curve $\gamma_R$ connecting $\partial B_{R}(p) \cap E_{R_0}^1$ and $\partial B_{R}(p) \cap E_{R_0}^2$ that minimizes the weighted length $L_{u_R}^\alpha$. Note that $\gamma_R \cap B_{R_0}(p) \neq \emptyset$ for all $R$ and a fixed large number $R_0$.

    Taking $R\to \infty$ and passing to a subsequence, $\gamma_R$ converges locally and smoothly to $\tilde{\gamma}$, which passes through the closure of $B_{R_0}(p)$.
	Each component of $\tilde{\gamma}$ is $L_u^\alpha$-minimizing and has infinite $g$-length.
	We then select one of the connected components of $\tilde{\gamma}$ that passes through the closure of $B_{R_0}(p)$. This is the desired weighted minimizing line.
\end{proof}

\section{Nonnegative spectral Ricci at infinity implies finite ends}\label{section3}

This section is devoted to proving the following result.
\begin{theorem}\label{classic-finite-ends}
    Let $M$ be an $n$-dimensional complete noncompact Riemannian manifold. Suppose $K$ is a compact subset of $M$ such that the spectral $\alpha$-Ricci curvature of $M$ is nonnegative on $M\setminus K$ for $\alpha<\frac{4}{n-1}$ and $M$ has inﬁnitely many ends. Then there exists a weighted minimizing geodesic line outside a larger compact subset $\tilde{K}$.
\end{theorem}

    Let $M$ be a noncompact manifold, and let $\Omega \subset M$ be a compact subset such that $M \setminus \Omega$ has finitely many connected components. An \textit{end} $E$ of $M$ (with respect to $\Omega$) is an unbounded connected component of $M \setminus \Omega$.

    We say $M$ has \textit{finitely many ends} if there exists a compact subset $\Omega_0 \subset M$ such that, for every compact $\Omega \supset \Omega_0$, the number of unbounded components of $M \setminus \Omega$ remains the same. For any compact set $\Omega$, $M\setminus \Omega$ may contain bounded components but these are not ends. We remark here that we always choose compact set $\Omega$ such that components of $M\setminus\Omega$ have smooth boundary throughout this article.

\begin{lemma}
    There exists an unbounded component of $M\setminus \Omega$, denoted by $\tilde{M}$, and infinitely many $L_u^\alpha$-minimizing geodesic rays on $\tilde{M}$ that are perpendicular to $\partial\tilde{M}$.
\end{lemma}

\begin{proof}
    Since the number of ends of $M$ is not bounded, there must exist at least one unbounded component of $M\setminus \Omega$ which has infinitely many ends; we denote it by $\tilde{M}$. Since all the ends of $\tilde{M}$ are interior ends, by \cite[Proposition 4.1]{hong-wang-jfa}, for each end of $\tilde{M}$, there is a corresponding free boundary $L_u^\alpha$-minimizing geodesic ray. 
\end{proof}

Each $L_u^\alpha$-minimizing geodesic ray $\gamma_i$ has an intersection point $p_i$ with $\partial \tilde{M}$. Because $\partial\tilde{M}$ is compact, there is a subsequence of $\{p_i\}$, still denoted by $\{p_i\}$, that converges to a fixed point $p\in \partial \tilde{M}$.  Theorem \ref{classic-finite-ends} follows from the following result.

\begin{lemma}
    For any compact subset $\tilde{K}\subset \tilde{M}$ containing $\partial\tilde{M}$, there exists an $L_u^\alpha$-minimizing line in the interior of $\tilde{M}\setminus \tilde{K}$.
    \label{lemma:exists-minimizing-line}
\end{lemma}
\begin{proof}
    Since $p$ is an accumulation point, there is a sequence of free boundary $L_{u}^\alpha$-minimizing rays $\gamma_{i}$ (corresponding to $p_i$) converging smoothly and locally to a free boundary $L_{u}^\alpha$-minimizing ray $\gamma_p$. Thus for any $\epsilon>0$ and any compact subset $\tilde{K}$ in $\tilde{M}$ containing $\partial \tilde{M}$, we can find $i=i(\epsilon,\tilde{K})\in \mathbb{N}$ such that $\gamma_i$ satisfies
    \begin{equation}\label{epsiloninequality}
        L_u^\alpha(\gamma_i\cap \partial_{in} \tilde{K},\gamma_p\cap \partial_{in} \tilde{K})<\epsilon
    \end{equation}
    where $\partial_{in} \tilde{K}$ denotes the boundary portion of $\partial \tilde{K}$ that lies in the interior of $\tilde{M}$.
    
    Let $q_{-j}$, $j\in \mathbb{Z}$,  be a sequence of points on $\gamma_i$ outside $\tilde{K}$ that diverge to infinity. Let $q_j$, $j\in \mathbb{Z}$, be a sequence of points on $\gamma_p$ outside $\tilde{K}$ that diverge to infinity. Denote $\tilde{M}_n:=\{x\in\tilde{M}:\operatorname{dist}(x,\partial\tilde{M})\leq n)\}$. 
    
    For fixed $j$, choose $n$ large enough such that $\tilde{M}_n$ contains $q_j$ and $q_{-j}$. We then choose a perturbation of the weighted function $u_n$ such that:
    \begin{itemize}
        \item $u_n = u$ in $\tilde{M}_n$,
        \item $u_n \geq u$ on $M$,
        \item $u_n > \max\{1, u\}$ in $M \setminus \tilde{M}_{n+1}$.
    \end{itemize}
    
    Let $\tau^n_j$ be the connected $L_{u_n}^\alpha$-minimizing geodesic curve in $\tilde{M}$ connecting $q_{-j}$ and $q_j$. We claim that for any $j\in \mathbb{Z}$, $\tau^n_j$ is smooth and is disjoint from $\partial \tilde{M}$. Indeed, if it meets $\partial \tilde{M}$, then
    \begin{equation}\label{triangleinequality}
        L_{u_n}^\alpha(\tau^n_j)\geq L_{u_n}^\alpha(q_{-j},\partial\tilde{M})+L_{u_n}^\alpha(q_{j},\partial\tilde{M}).
    \end{equation}
    Let $\gamma_i(q_{-j},\gamma_i\cap \partial_{in}\tilde{K})$ be the portion of $\gamma_i$ that connects $q_{-j}$ and  $\gamma_i\cap \partial_{in}\tilde{K}$. According to \eqref{epsiloninequality}, there exists a small arc $s_j$ connecting $\gamma_i\cap \partial_{in} \tilde{K}$ and $\gamma_p\cap \partial_{in} \tilde{K} $ such that 
    \[L_{u_n}^\alpha(s_j)<\epsilon.\]
    We can denote
    \[\tilde{\tau}^n_j=\gamma_i(q_{-j},\gamma_i\cap \partial_{in}\tilde{K}) \cup s_j \cup \gamma_p(\gamma_p\cap \partial_{in}\tilde{K}).\]
    Then \begin{align*}
        L_{u_n}^\alpha(\tilde{\tau}^n_j)<L_{u_n}^\alpha(q_{-j},\partial\tilde{M})+L_{u_n}^\alpha(q_{j},\partial\tilde{M})+\epsilon-2L_{u_n}^\alpha(\partial_{in}\tilde{K},\partial \tilde{M}).
    \end{align*}
    If $\epsilon<2L_{u_n}^\alpha(\partial_{in}\tilde{K},\partial \tilde{M})$ (which we ensure by choosing $\epsilon$ sufficiently small and $i$ sufficiently large), then by \eqref{triangleinequality} we obtain 
    \[L_{u_n}^\alpha(\tilde{\tau}^n_j)<L_{u_n}^\alpha(\tau^n_j)\]
    contradicting the minimality of $\tau_j^n$. Thus $\tau_j^n$ is connected smooth curve. By taking $n\rightarrow \infty,$ we obtain a smooth curve $\tau_j$ that is $L_u^\alpha$-minimizing. If there is a component of $\tau_j$ which is a line, we are done. Otherwise, since $\tau_j$ extends into two ends, it must intersect a compact subset. Thus letting $j$ go to infinity, we obtain a smooth limit $\tau$ that is $L_u^\alpha$-minimizing. Note there may be multiple components in $\tau$, but each of them is an $L_u^\alpha$-minimizing line. This completes the proof.
\end{proof}

We now start to prove the main theorem. First we show

\begin{proposition}\label{perturbation}
    There exists a positive function $u$ on $\tilde{M}$ such that 
    \[-\alpha\Delta u+\Ric u>0 \ \text{on}\ \tilde{M}\setminus \tilde{K}.\]
\end{proposition}
\begin{proof}
    The proof is contained in \cite[Lemma 2.2]{antonelli-pozzetta-xu} and \cite[equation (2.2)]{antonelli-pozzetta-xu}. Since the proof is local, the boundary of $\tilde{M}$ will not affect the result.
\end{proof}

\begin{proof}[Proof of Theorem \ref{classic-finite-ends}]
Suppose $M$ has infinitely many ends. By Lemma \ref{lemma:exists-minimizing-line}, there exists a weighted minimizing line $\tau$ outside a compact set.
By the second variation formula of $L_u^\alpha$ (see e.g. Proposition 2.9 in \cite{hong-wang-jfa}), when $\alpha<\frac{4}{n-1}$, we have
 \begin{align}\label{the key inequality-line}
 c_1(n,\alpha,\varepsilon)\int_\tau \psi^2_s ds&\geq  c_2(n,\alpha,\varepsilon)\int_\tau u^{-2}\psi^2u_s^2 ds\nonumber\\
         &+\int_\tau\left(\operatorname{Ric}(\partial_s,\partial_s)-\alpha\frac{\Delta u}{u}\right)\psi^2 ds+\alpha\int_\tau\psi^2|\nabla^\perp\log u|^2 ds.
     \end{align}
     Plugging in the standard test function, we obtain $\operatorname{Ric}(\partial_s,\partial_s)-\alpha\frac{\Delta u}{u}=0$ along $\tau$, contradicting Proposition \ref{perturbation}.
\end{proof}

\end{document}